\pdfoutput=1
\documentclass[11pt,a4paper]{article}

\usepackage[T1]{fontenc}
\usepackage[utf8]{inputenc}
\usepackage{lmodern}
\usepackage{microtype}
\usepackage[a4paper,margin=1.15in]{geometry}

\usepackage{amsmath,amssymb,amsfonts,amsthm}
\usepackage{mathtools}
\usepackage{graphicx}
\DeclareGraphicsExtensions{.pdf,.PDF,.eps,.EPS,.png,.PNG,.jpg,.JPG,.jpeg,.JPEG}
\usepackage{enumerate}

\usepackage{xurl}
\usepackage[hidelinks]{hyperref}
\hypersetup{
  pdftitle={The modified boundary contact process: invariant measures and critical extinction},
  pdfauthor={Célio Terra},
  pdfsubject={Probability theory},
  pdfkeywords={modified boundary contact process; invariant measure; interacting particle systems}
}

\makeatletter
\def\@MRExtract#1 #2!{#1}
\newcommand{\MR}[1]{%
  \xdef\@MRSTRIP{\@MRExtract#1 !}%
  \href{https://mathscinet.ams.org/mathscinet-getitem?mr=\@MRSTRIP}{MR\@MRSTRIP}}

\makeatother

\theoremstyle{plain}
\newtheorem{theorem}{Theorem}[section]

\newtheorem{lemma}[theorem]{Lemma}

\newtheorem{proposition}[theorem]{Proposition}

\theoremstyle{definition}

\theoremstyle{remark}
\newtheorem{remark}[theorem]{Remark}

\numberwithin{equation}{section}

\title{The modified boundary contact process:\\
invariant measures and critical extinction}
\author{Célio Terra\\[0.4em]
\small Universidade Federal de Minas Gerais, Brasil\\
\small Current affiliation: Universidade Federal do Rio de Janeiro, Brasil\\
\small \texttt{caugusto.terra@gmail.com}}
\date{}

\begin{document}

\maketitle
\thispagestyle{plain}

\begin{abstract}
\noindent
We study the one-dimensional modified boundary contact process, in which infections across the two boundary edges of the infected region occur at rate \(\lambda_e\), while all other infections occur at rate \(\lambda_i\). We prove two results in different parts of the phase diagram. First, in the non-attractive region $\lambda_e>\lambda_i\geq\lambda_c$, where $\lambda_c$ is the critical parameter of the standard contact process, the process seen from its rightmost infected site converges from every semi-infinite initial configuration to an invariant measure. Second, on the critical curve in the attractive region, the infection dies out almost surely.
\end{abstract}

\smallskip

\section{Introduction}
	The Harris contact process, introduced in~\cite{Harris74}, is a widely studied model for the spread of infections. Each site in $\mathbb{Z}$ can be either \emph{infected} or \emph{healthy}. Each infected site becomes healthy at rate $1$ and attempts to infect each nearest neighbour at rate $\lambda>0$; if the chosen neighbour is healthy, it becomes infected. We assume familiarity with the basic concepts and results about the contact process, as presented for instance in~\cite[Chapter VI]{Liggett05} or in~\cite{Valesin2024}.

	Many modifications of the basic contact process have been proposed. One such modification, introduced by Durrett and Schinazi in~\cite{DurrettSchinazi00}, is the \emph{contact process with modified boundary}, which introduces two parameters, $(\lambda_i, \lambda_e)$. Infections to the right of the rightmost site and to the left of the leftmost site occur at rate $\lambda_e$, while infections at other sites occur at rate $\lambda_i$. Thus, there is an \emph{internal} infection rate, $\lambda_i$, and an \emph{external} one, $\lambda_e$. This process appears in the literature under slightly different names. Durrett and Schinazi~\cite{DurrettSchinazi00} use the term \emph{boundary modified contact process}, while Andjel and Rolla~\cite{AndjelRolla23} refer to related formulations as contact processes with enhancements. In this paper we use the name \emph{modified boundary contact process} throughout.
	
	Given a measure $\mu$ on the power set of $\mathbb{Z}$, we denote by $\xi^{\mu}=(\xi_t^{\mu})_{t \ge 0}$ a contact process with parameters $(\lambda_i, \lambda_e)$ and initial condition $\xi^{\mu}_0$ sampled from $\mu$. We denote by $\Psi \xi^{\mu}$ the process $\xi^\mu$ seen from the right edge, meaning the contact process translated so that its rightmost particle is located at the origin. For a subset $A$ of $\mathbb{Z}$, we write $\xi^A$ and $\Psi \xi^A$ for $\xi^{\delta_A}$ and $\Psi \xi^{\delta_A}$, respectively. A formal definition is provided in Section~\ref{section_descriptioncontact}.
	
	In this work, we analyze two questions related to this process. The first concerns the existence of \emph{invariant measures}. In~\cite{Durrett84}, Durrett proved the existence of an invariant measure for the critical and supercritical classical contact process seen from the right edge. Galves and Presutti proved the convergence of the supercritical contact process seen from the edge to the invariant measure in~\cite{GalvesPresutti87}, and this result was extended to the critical case in~\cite{CoxDurrettSchinazi91}. Nonexistence of an invariant measure for the subcritical process was proved in~\cite{Schonmann87} for the discrete-time case and in~\cite{ASS90} for the continuous-time case. However, when conditioned on non-extinction, subcritical contact processes converge to a distribution supported on semi-infinite sets~\cite{AEGR15}.
	
	In the case of the modified boundary contact process, the following result was first proved in~\cite{AndjelRolla23}.
	Denote by $\theta(\lambda_i,\lambda_e)$ the probability that the modified boundary contact process survives forever when at time $0$ there is only one infected particle, and by $\lambda_c$ the critical parameter for the usual contact process on the line.
	\begin{theorem} \label{thm_andjelrolla}
		Suppose $\lambda_e \le \lambda_i$ and $\theta(\lambda_i, \lambda_e)>0$ or $\lambda_e=\lambda_i = \lambda_c$. There is a measure $\mu$ such that for every initial condition $A \subseteq \mathbb{Z}$ with $|A|=+\infty$ and $\sup A < +\infty$, $\Psi \xi^A_t \rightarrow \mu$ weakly. Moreover, $\Psi \xi^{\mu}_t \sim \mu$ for every $t \ge 0$.
	\end{theorem}
	
	When $\lambda_e \le \lambda_i$, the process is \emph{attractive}, meaning that adding more particles to the initial configuration helps the infection survive (see~\cite[Chapter III]{Liggett05} for a formal definition). This follows directly from the construction of the process given in Section~\ref{section_descriptioncontact}. In contrast, when $\lambda_i < \lambda_e$, the process is non-attractive. The tools and results available for studying non-attractive processes are far more limited than those for the attractive case.
	
	The main difference from the classical contact process is that, in the modified boundary process, whether an infection arrow is active depends on the current configuration. Indeed, one must know whether the infected site from which the arrow starts is a boundary site or an internal site, since the corresponding rates are $\lambda_e$ and $\lambda_i$, respectively. This is the point where the classical arguments cannot be used verbatim. In Section~\ref{section_invariantmeasure} we avoid this difficulty by introducing a right-boundary auxiliary process, and in Section~\ref{section_criticaldies} we use attractiveness to justify the restart argument in the block construction.
	
	In this work, we extend Theorem~\ref{thm_andjelrolla} to a non-attractive region. Namely, we treat the case $\lambda_e>\lambda_i\geq\lambda_c$. As in Theorem~\ref{thm_andjelrolla}, the result is stated for semi-infinite initial configurations. We do not treat finite initial configurations here; for such initial states, extinction has positive probability, so any analogue would require a different, for instance conditional, formulation. Our first main result is the following theorem.
	
	\begin{theorem} \label{thm_measure}
		Suppose $\lambda_e>\lambda_i\geq\lambda_c$. There is a measure $\tilde{\mu}$ such that for every initial condition $A \subseteq \mathbb{Z}$ with $|A|=+\infty$ and $\sup A < +\infty$, $\Psi \xi^A_t \rightarrow \tilde{\mu}$ weakly. Moreover, $\Psi \xi^{\tilde{\mu}}_t \sim \tilde{\mu}$ for every $t \ge 0$.
	\end{theorem}
	
	\begin{remark}
		It is natural to ask whether Theorem~\ref{thm_measure} also describes the process started from a finite non-empty set, after conditioning on non-extinction. This question is not addressed here. Such a result would require an additional argument showing that, under the conditioning, the infected set grows far enough to the left of its right edge so that the local dynamics near the right edge can be compared with the dynamics started from a semi-infinite configuration.
	\end{remark}
	
	The technique used in~\cite{AndjelRolla23} to prove invariance of the limiting distribution does not translate immediately to the non-attractive case. The proof of Theorem~\ref{thm_measure} follows the strategy of~\cite[Theorem 1]{CoxDurrettSchinazi91}, but with an auxiliary process whose enhancement acts only at the right boundary. This distinction is important: for semi-infinite configurations there is no left boundary, and therefore the original modified boundary contact process and this right-boundary auxiliary process have the same law after recentering at the right edge.
	
	\begin{remark}
	The proof uses survival of the right-boundary auxiliary process. When $\lambda_i=\lambda_c$ and $\lambda_e>\lambda_c$, this is the one-boundary survival result proved in~\cite[Remark~1.7]{AndjelRolla23}. When $\lambda_i>\lambda_c$, the same auxiliary process dominates the classical contact process with parameter $\lambda_i$, and hence survives. Thus the argument covers the whole region $\lambda_e>\lambda_i\geq\lambda_c$.
	\end{remark}

	Another key aspect of the modified boundary contact process is its phase space. The characteristics of the phase diagram were first studied in~\cite{DurrettSchinazi00}. Andjel and Rolla~\cite{AndjelRolla23} almost entirely characterized the phase diagram, but left open the question of survival on the critical curve in the attractive region. Although this second result concerns the attractive region rather than the non-attractive regime of Theorem~\ref{thm_measure}, it completes the remaining critical-curve part of the phase diagram. We prove the following result.
	
	\begin{theorem} \label{thm_criticaldies}
		Define
		\[\lambda^e_*(\lambda_i)=\inf\{\lambda_e : \theta(\lambda_i, \lambda_e)>0\}.\]
		Then, if $\lambda_i>\lambda_c$, $\theta(\lambda_i, \lambda^e_*(\lambda_i))=0$.
	\end{theorem}
	
	By~\cite[Theorems 4 and 5]{AndjelRolla23}, we have that $\lambda_*^e(\lambda_i)< \lambda_i$ when $\lambda_i>\lambda_c$. Since the process is attractive when $\lambda_i \ge \lambda_e$, the proof of Theorem~\ref{thm_criticaldies} closely follows the approach used for the standard critical contact process in~\cite{BG90}. Theorem~\ref{thm_criticaldies} completes the characterization of the phase space for the modified boundary contact process, illustrated in Figure~\ref{fig_phasecontact}.

	\begin{figure}
		\centering 
		\includegraphics[width=0.7\linewidth]{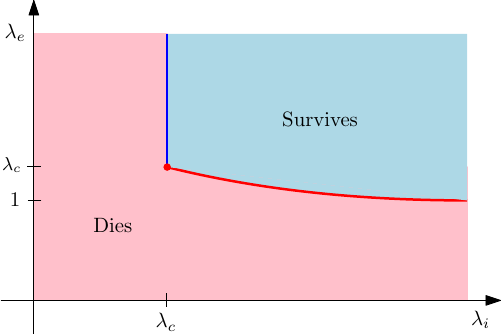}%
		\caption{Phase space for the contact process with modified boundary}%
		\label{fig_phasecontact}%
	\end{figure}

	This note is structured as follows. Section~\ref{section_descriptioncontact} is dedicated to the construction of the modified boundary contact process. Theorem~\ref{thm_measure} is proved in Section~\ref{section_invariantmeasure} and Theorem~\ref{thm_criticaldies} is proved in Section~\ref{section_criticaldies}.
	
	This note is based on the second chapter of the author's PhD thesis~\cite{Terra2024}.
	
	\section{Description of the model} \label{section_descriptioncontact}
	In this section we construct the modified boundary contact process and fix the notation used throughout the paper. The process itself is denoted by $\xi$; the symbol $\zeta$ is reserved for auxiliary processes.
	
	A \emph{configuration} is a subset of $\mathbb{Z}$. We denote by $\Sigma$ the set of all configurations, identified with $\{0,1\}^{\mathbb{Z}}$ and endowed with the product topology. For $A\subseteq\mathbb{Z}$, let $\mathcal{R}A=\sup A$ and $\mathcal{L}A=\inf A$. We also define $
	\Sigma^{\ominus}=\{A\in\Sigma: |A|=+\infty,\ \mathcal{R}A<+\infty\}$ and $\Sigma^{\odot}=\{A\in\Sigma^{\ominus}:\mathcal{R}A=0\}$. If $A\in\Sigma^{\ominus}$, the configuration seen from its right edge is $\Psi A=\{x-\mathcal{R}A:x\in A\}$.
	Thus $\Psi A\in\Sigma^{\odot}$.
	
	We also fix the following notation for cylinder sets in $\Sigma^{\odot}$. If $L<+\infty$ and $\kappa\in\Sigma^{\odot}$, let $C(L,\kappa)={D\in\Sigma^{\odot}:D\cap[-L,0]=\kappa\cap[-L,0]}$. We call such a set a cylinder. These cylinders generate the topology of $\Sigma^{\odot}$.

	We first describe the dynamics for finite configurations. If the current state is $B\neq\emptyset$, each infected site dies at rate $1$. An infected site $x\in B$ gives birth to $x+1$ at rate $\lambda_e$ if $x=\mathcal{R}B$, and at rate $\lambda_i$ otherwise. Similarly, $x$ gives birth to $x-1$ at rate $\lambda_e$ if $x=\mathcal{L}B$, and at rate $\lambda_i$ otherwise. If the target site is already infected, the birth does not change the configuration.
	
	We use the usual graphical construction, in an enlarged form that allows us to couple different values of the parameters. For each $x\in\mathbb{Z}$, let $\omega_x$ be an independent Poisson point process on $(0,+\infty)$ with intensity $1$; its marks are recovery marks. For each oriented nearest-neighbor edge $(x,x\pm1)$, let $\omega_{x,x\pm1}$ be an independent Poisson point process on $(0,+\infty)\times(0,+\infty)$ with intensity $1$. We say that the arrow from $(x,t)$ to $(x\pm1,t)$ is $\lambda$-open if
	\[
	\omega_{x,x\pm1}\cap(\{t\}\times[0,\lambda))\neq\emptyset.
	\]
	The probability measure induced by all these Poisson processes is denoted by $\mathbf{P}$.
	
	Given a finite initial condition $A$, the process $(\xi^A_{s,t})_{t\ge s}$ is obtained from this graphical construction by setting $\xi^A_{s,s}=A$ and applying the transition rules above. Equivalently, at time $t$ an arrow from $x$ to $x+1$ is used with parameter $\lambda_e$ only when $x$ is the current right edge, and with parameter $\lambda_i$ otherwise; the analogous rule holds on the left. If $s=0$, we write $\xi_t^A$ instead of $\xi^A_{0,t}$.
	
	For $A\in\Sigma^{\ominus}$, we define $\xi^A_{s,t}$ by the standard finite-volume approximation. Namely, take finite sets $A_n\uparrow A$ with $\sup A_n=\mathcal{R}A$ for all large $n$, construct $\xi^{A_n}_{s,t}$ with the same graphical representation, and let $n\to\infty$. Since only finitely many Poisson marks can influence a fixed finite space-time region, this gives a well-defined process on every finite time interval. We write $\xi_t^-$ when $A=\mathbb{Z}^-$. The law of the process with parameters $(\lambda_i,\lambda_e)$ is denoted by $\mathbb{P}_{\lambda_i,\lambda_e}$; the subscript is omitted when the parameters are clear.
	
	A function $f:\Sigma\to\mathbb{R}$ is \emph{increasing} if $f(A)\le f(B)$ whenever $A\subseteq B$. Given two random configurations $\xi$ and $\eta$, we say that $\xi$ stochastically dominates $\eta$, and write $\xi\succcurlyeq\eta$, if $\mathbb{E}[f(\xi)]\ge\mathbb{E}[f(\eta)]$ for every bounded increasing function $f$. From the graphical construction, the modified boundary contact process is attractive when $\lambda_e\le\lambda_i$.
	
	We will use the following notion of active path. Fix an initial configuration $A$ and times $s<t$. A path from $(x,s)$ to $(y,t)$ is said to be active if it avoids recovery marks and each of its jumps follows an infection arrow that is allowed for the process started from $A$. More precisely, an arrow from $z$ to $z+1$ at time $u$ is allowed with parameter $\lambda_e$ if $z=\mathcal{R}\xi^A_{s,u^-}$, and with parameter $\lambda_i$ otherwise. The analogous rule holds for arrows from $z$ to $z-1$, using the left edge. We write $(x,s)\rightsquigarrow(y,t)$ when such an active path exists.

	\section{Convergence to invariant measure} \label{section_invariantmeasure}
	
	This section is devoted to the proof of Theorem~\ref{thm_measure}. Throughout the section we assume that $\lambda_e>\lambda_i\geq\lambda_c$.
	
	The first step is to show that the process seen from the right edge forgets its initial condition, uniformly over semi-infinite initial configurations.
	
	\begin{proposition} \label{prop_uniquenessmeasure}
		For every cylinder set $C\subseteq\Sigma^{\odot}$,
		\begin{equation*}
			\lim_{t\to+\infty}
			\left|
			\mathbb{P}_{\lambda_i,\lambda_e}(\Psi\xi_t^A\in C)
			-
			\mathbb{P}_{\lambda_i,\lambda_e}(\Psi\xi_t^B\in C)
			\right|=0,
		\end{equation*}
		uniformly in $A,B\in\Sigma^{\odot}$.
	\end{proposition}
	
	The second step is to prove that any subsequential limit is supported on $\Sigma^{\odot}$.
	
	\begin{proposition}\label{prop_fullmeasure}
		Let $\tilde{\mu}$ be a subsequential weak limit of $(\Psi\xi_t^A)_{t\geq0}$ for some $A\in\Sigma^{\odot}$. Then $\tilde{\mu}(\Sigma^{\odot})=1$.
	\end{proposition}
	
	We now prove Theorem~\ref{thm_measure} assuming Propositions~\ref{prop_uniquenessmeasure} and~\ref{prop_fullmeasure}.
	
	\begin{proof}[Proof of Theorem~\ref{thm_measure}]
		By translation invariance, it is enough to consider initial configurations in $\Sigma^{\odot}$. Fix $A\in\Sigma^{\odot}$. By compactness of $\Sigma$, there is a sequence $(t_k)_k$ with $t_k\to+\infty$ such that $\Psi\xi_{t_k}^A$ converges weakly to a probability measure $\tilde{\mu}$.
		
		We first show that convergence holds for all times, not only along the subsequence. Let $C$ be a cylinder set. Given $t>0$, choose $k=k(t)$ such that $t_k\leq t$ and $t_k\to+\infty$ as $t\to+\infty$. Denote by $\nu^A_{t-t_k}$ the law of $\Psi\xi^A_{t-t_k}$. Since $(\Psi\xi_t^B)_{t\geq0}$ is a Markov process, the Chapman--Kolmogorov equations give
		\begin{align*}
			&\left|\tilde{\mu}(C)-\mathbb{P}_{\lambda_i,\lambda_e}(\Psi\xi_t^A\in C)\right|\\
			&\quad\leq
			\left|\tilde{\mu}(C)-\mathbb{P}_{\lambda_i,\lambda_e}(\Psi\xi_{t_k}^A\in C)\right|\\
			&\qquad+
			\left|\int
			\left[
			\mathbb{P}_{\lambda_i,\lambda_e}(\Psi\xi_{t_k}^A\in C)
			-
			\mathbb{P}_{\lambda_i,\lambda_e}(\Psi\xi_{t_k}^B\in C)
			\right]
			\,d\nu^A_{t-t_k}(B)\right|.
		\end{align*}
		The first term goes to zero by the choice of the subsequence, and the second one goes to zero by Proposition~\ref{prop_uniquenessmeasure}. Thus $\Psi\xi_t^A$ converges weakly to $\tilde{\mu}$.
		
		The limit does not depend on the initial condition. Indeed, if $\Psi\xi_t^A$ converges to $\tilde{\mu}$ and $\Psi\xi_t^B$ converges to $\tilde{\nu}$, then Proposition~\ref{prop_uniquenessmeasure} implies that $\tilde{\mu}$ and $\tilde{\nu}$ agree on all cylinder sets. Hence $\tilde{\mu}=\tilde{\nu}$.
		
		It remains to prove invariance. By Proposition~\ref{prop_fullmeasure}, $\tilde{\mu}$ is supported on $\Sigma^{\odot}$. For every fixed $s\geq0$ and every bounded continuous function $f:\Sigma\to\mathbb{R}$, the map $D\mapsto\mathbb{E}_{\lambda_i,\lambda_e}[f(\Psi\xi_s^D)]$ is continuous at every $D\in\Sigma^{\odot}$. Thus the process seen from the right edge is Feller on $\Sigma^{\odot}$. Since $\Psi\xi_t^A\Rightarrow\tilde{\mu}$ and the limit is independent of $A$, Proposition~I.1.8 of~\cite{Liggett05} gives $\Psi\xi_s^{\tilde{\mu}}\sim\tilde{\mu}$ for every $s\geq0$.
	\end{proof}
	
	The proof of Proposition~\ref{prop_uniquenessmeasure} uses the construction of Galves and Presutti~\cite{GalvesPresutti87}, adapted to the present non-attractive setting. The main difficulty is that the classical coupling between processes started from a single site and from $\mathbb{Z}^-$ does not directly give equality on a fixed window behind the right edge. We replace it by an auxiliary process whose enhanced rate $\lambda_e$ is used only for infections from the current rightmost site to its right neighbour, while all other infection arrows are used at rate $\lambda_i$. For semi-infinite initial configurations this auxiliary process has the same law as the original modified boundary contact process seen from the right edge, because such configurations have a right boundary but no left boundary.
	
	\begin{proof}[{Proof of Proposition~\ref{prop_uniquenessmeasure}.}]
	Let $(\zeta_t)_t$ be the auxiliary process with infection rate $\lambda_e$ only from the right edge to the right, and infection rate $\lambda_i$ everywhere else. More precisely, infections from $\mathcal{R}\zeta_t$ to $\mathcal{R}\zeta_t+1$ occur at rate $\lambda_e$, while all other nearest-neighbor infections occur at rate $\lambda_i$. We denote the law of this process by $\widehat{\mathbb{P}}_{\lambda_i,\lambda_e}$.
			
	Let $\tau_0:=0$ and, for $k\ge1$, let $\tau_k$ be the first time after $\tau_{k-1}$ at which the copy of the auxiliary process started from the origin at time $\tau_{k-1}$ dies out. More explicitly,
		\[\tau_k:=\inf\{t>\tau_{k-1}:\zeta_{\tau_{k-1},t}^{0}=\emptyset\}.\]
	Take $A\in\Sigma^{\odot}$. We define a recentered version $\tilde{\zeta}^A$ recursively on the intervals $[\tau_{k-1},\tau_k)$ by
		\[\tilde{\zeta}^A_t=\zeta_{\tau_{k-1},t}^{\Psi \eta}+\mathcal{R}\eta,\qquad \eta=\tilde{\zeta}^A_{\tau_{k-1}}.\]
	In words, at each time $\tau_{k-1}$ we recenter the current configuration at its right edge, restart the auxiliary process from that recentered configuration using the same graphical construction, and then translate the resulting configuration back to the previous right-edge position. By translation invariance of the graphical construction, $(\Psi\tilde{\zeta}^A_t)_{t\ge0}$ has the same distribution as $(\Psi\zeta^A_t)_{t\ge0}$ under $\widehat{\mathbb{P}}_{\lambda_i,\lambda_e}$. Since $A\in\Sigma^{\odot}$ is semi-infinite to the left, $(\Psi\zeta^A_t)_{t\ge0}$ under $\widehat{\mathbb{P}}_{\lambda_i,\lambda_e}$ has the same distribution as $(\Psi\xi^A_t)_{t\ge0}$ under $\mathbb{P}_{\lambda_i,\lambda_e}$.
	
	Define $N(t)=\sum_{k=1}^{\infty}\mathbf{1}_{\{\tau_k\le t\}}$. Fix $L<+\infty$, $\kappa\in\Sigma^{\odot}$, and consider the cylinder $C=\{D\in\Sigma^{\odot}:D\cap[-L,0]=\kappa\cap[-L,0]\}$. Every finite-dimensional cylinder for the process seen from the right edge can be written in this form for some $L$ and $\kappa$.
	
	For every $A,B\in\Sigma^{\odot}$,
	\begin{equation} \label{eq_difference_cylinder_bound}
		\begin{aligned}
			&|\mathbb{P}_{\lambda_i,\lambda_e}(\Psi\xi_t^A\in C)
			-\mathbb{P}_{\lambda_i,\lambda_e}(\Psi\xi_t^B\in C)|\\
			&\quad=
			|\widehat{\mathbb{P}}_{\lambda_i,\lambda_e}(\Psi\tilde{\zeta}_t^A\in C)
			-\widehat{\mathbb{P}}_{\lambda_i,\lambda_e}(\Psi\tilde{\zeta}_t^B\in C)|\\
			&\quad\leq
			\widehat{\mathbb{P}}_{\lambda_i,\lambda_e}
			\big(\Psi\tilde{\zeta}_t^A\cap[-L,0]
			\neq
			\Psi\tilde{\zeta}_t^B\cap[-L,0]\big).
		\end{aligned}
	\end{equation}
	
	Let $s=\tau_{N(t)}$.  The only point which differs from the classical
	contact process is that the set of allowed arrows may depend on the
	current configuration.  Here this difficulty is removed by the
	one-boundary auxiliary construction: after the recentering time $s$,
	both $\widetilde\zeta^A$ and $\widetilde\zeta^B$ are run from right edge
	$0$ with the same graphical representation, and the descendants of the
	site $0$ use exactly the arrows of the auxiliary process
	$\zeta^0_{s,t}$.  Thus the usual one-dimensional no-crossing argument
	applies to these active paths.  In particular,
	\[
	\left\{
	\Psi\widetilde\zeta^A_t\cap[-L,0]
	\neq
	\Psi\widetilde\zeta^B_t\cap[-L,0]
	\right\}
	\subseteq
	\left\{
	\inf\Psi\zeta^0_{\tau_{N(t)},t}>-L
	\right\}.
	\]
	Indeed, on the complement of the event on the right, the descendants of
	$0$ at time $s$ separate the window $[-L,0]$ from all sites initially to
	the left of $0$.
	
	By definition of $N(t)$, $\zeta^0_{\tau_{N(t)},t}\neq\varnothing$.
	Therefore $\{\inf\Psi\zeta^0_{\tau_{N(t)},t}>-L\}$ is contained in
	$\{0<|\zeta^0_{\tau_{N(t)},t}|\le L\}$, and hence
	\begin{equation}
		\widehat{\mathbb P}_{\lambda_i,\lambda_e}
		\left(
		\Psi\widetilde\zeta^A_t\cap[-L,0]
		\neq
		\Psi\widetilde\zeta^B_t\cap[-L,0]
		\right)
		\le
		\widehat{\mathbb P}_{\lambda_i,\lambda_e}
		\left(
		0<|\zeta^0_{\tau_{N(t)},t}|\le L
		\right).
		\tag{3.2}
		\label{eq_coupling_error}
	\end{equation}
	We now justify that the right-hand side of~\eqref{eq_coupling_error} tends to zero. Fix $L<+\infty$. There exists $\delta_L>0$ such that, for every non-empty finite set $F$ with $|F|\leq L$, the auxiliary process started from $F$ dies out during the next unit time with probability at least $\delta_L$. Indeed, this occurs if all particles in $F$ recover before time $1$ and no infection arrow issued from them is used before their recovery.
	
	Let $\rho=\widehat{\mathbb{P}}_{\lambda_i,\lambda_e}(\zeta_t^0\neq\emptyset\text{ for all }t\geq0)>0$. By the Markov property,
	\[
	\delta_L\,
	\widehat{\mathbb{P}}_{\lambda_i,\lambda_e}
	\big(0<|\zeta_s^0|\leq L\big)
	\leq
	\widehat{\mathbb{P}}_{\lambda_i,\lambda_e}
	\big(\zeta_s^0\neq\emptyset,\ \zeta_{s+1}^0=\emptyset\big).
	\]
	The right-hand side tends to zero as $s\to+\infty$, since $\widehat{\mathbb{P}}_{\lambda_i,\lambda_e}(\zeta_s^0\neq\emptyset)$ decreases to $\rho$. Hence
	\[
	\widehat{\mathbb{P}}_{\lambda_i,\lambda_e}
	\big(|\zeta_s^0|\leq L\,\big|\,\zeta_s^0\neq\emptyset\big)
	\longrightarrow0
	\quad\text{as }s\to+\infty.
	\]
	Since, after the first surviving trial, $N(t)$ is constant and $t-\tau_{N(t)}\to+\infty$, the right-hand side of~\eqref{eq_coupling_error} tends to zero. The convergence is uniform in $A$ and $B$, and the proposition follows
	\end{proof}
	
	The proof of Proposition~\ref{prop_fullmeasure} follows the approach of~\cite{ASS90}. Since the deterministic velocity of the right edge is not available for the non-attractive process used in Theorem~\ref{thm_measure}, we compare it with the classical contact process with parameter $\lambda_i$.
	
	\begin{proof}[{Proof of Proposition~\ref{prop_fullmeasure}}.]
	Let $\xi^-_t$ be the modified boundary contact process with parameters $(\lambda_i,\lambda_e)$ and initial condition $\mathbb{Z}^-$. Let $\zeta^-_t$ be the classical contact process with parameter $\lambda_i$ and the same initial condition. Denote by $\tilde{\mu}_t$ the law of $\Psi\xi^-_t$. For fixed $i,j\in\mathbb{N}$, define
		\[
		A_{i,j}=\{B\in\Sigma:\mathcal{R}B=0 \text{ and } |B\cap[-i,0]|<j\}.
		\]
	
	For every $k\in\mathbb{N}$, the positive part of the increment $\mathcal{R}\xi^-_{k+1}-\mathcal{R}\xi^-_k$ is stochastically dominated by a Poisson random variable with parameter $\lambda_e$, because the right edge can move to the right only through births from the current rightmost site. Thus,
		\begin{equation} \label{eq_expectedgepositive}
			\mathbb{E}_{\lambda_i,\lambda_e}[\mathcal{R}\xi^{-}_{k+1}-\mathcal{R}\xi^{-}_{k}]^{+} \le \lambda_e.
		\end{equation}
	If $\Psi\xi^-_k\in A_{i,j}$, then there are fewer than $j$ infected sites in the interval $[\mathcal{R}\xi^-_k-i,\mathcal{R}\xi^-_k]$. With probability at least $p(j)>0$, depending only on $j$ and on the rates, all those infected sites recover before producing any birth that could keep the right edge inside this interval. On that event, the right edge moves at least $i$ sites to the left during $[k,k+1]$. Hence,
		\begin{equation} \label{eq_boundexpectedgenegative}
			\mathbb{E}_{\lambda_i,\lambda_e}[\mathcal{R}\xi^{-}_{k+1}-\mathcal{R}\xi^{-}_{k}]^{-} \ge ip(j)\tilde{\mu}_k(A_{i,j}).
		\end{equation}
	Multiplying~\eqref{eq_boundexpectedgenegative} by $-1$ and summing with~\eqref{eq_expectedgepositive},
		\begin{equation} \label{eq_boundexpecedge}
			\mathbb{E}_{\lambda_i,\lambda_e}[\mathcal{R}\xi^{-}_{k+1}-\mathcal{R}\xi^{-}_{k}] \le \lambda_e-ip(j)\tilde{\mu}_k(A_{i,j}).
		\end{equation}
	
We couple $(\xi^-_t)_t$ and $(\zeta^-_t)_t$ using the same recovery
marks and the same $\lambda_i$-open infection arrows.  This comparison
does not use attractiveness of the modified boundary process; it is a
pathwise comparison with the classical contact process started from the
same initial condition $\mathbb Z^-$.  Whenever a $\lambda_i$-open arrow
is used by $\zeta^-$, the same arrow is also available to $\xi^-$: if
its initial point is internal for $\xi^-$, it is used at rate
$\lambda_i$, while if it is the right boundary, it is allowed at rate
$\lambda_e>\lambda_i$.  Since the two processes start from the same
configuration, induction over the graphical marks gives	\begin{equation} \label{eq_classical_inside_modified}
		\zeta^-_t\subseteq \xi^-_t
		\qquad\text{for all }t\geq0.
	\end{equation}
Summing~\eqref{eq_boundexpecedge} from $k=0$ to $n-1$ gives
	\begin{equation} \label{eq_boundedge_n}
		\mathbb{E}_{\lambda_i,\lambda_e}[\mathcal{R}\xi_n^-]
		\leq n\lambda_e-ip(j)\sum_{k=0}^{n-1}\tilde{\mu}_k(A_{i,j}).
	\end{equation}
	On the other hand, by~\eqref{eq_classical_inside_modified}, $\mathcal{R}\zeta_n^-\leq\mathcal{R}\xi_n^-$ for every $n$. Hence, by~\eqref{eq_boundedge_n},
	\begin{equation} \label{eq_integraledge}
		\frac{1}{n}\sum_{k=0}^{n-1}\tilde{\mu}_k(A_{i,j})
		\leq
		\frac{\lambda_e}{ip(j)}
		-
		\frac{\mathbb{E}_{\lambda_i}[\mathcal{R}\zeta_n^-]}{ip(j)n}.
	\end{equation}
	Applying Proposition~\ref{prop_uniquenessmeasure},
		\begin{equation} \label{eq_convergenceAij}
			\lim_{n\to\infty} \frac{1}{n}\sum_{k=0}^{n-1} \tilde{\mu}_k(A_{i,j}) = \tilde{\mu}(A_{i,j}).
		\end{equation}
	Moreover, by~\cite[Theorem VI.2.19]{Liggett05}, the right edge of the classical contact process satisfies $\mathcal{R}\zeta_n^-/n\to\alpha(\lambda_i)$ almost surely and in $L^1$, for some finite constant $\alpha(\lambda_i)$. Taking the limit superior as $n\to\infty$ in~\eqref{eq_integraledge} and using~\eqref{eq_convergenceAij}, we obtain
	\begin{equation}\label{eq_measureAij}
	\tilde{\mu}(A_{i,j})
	\leq
	\frac{\lambda_e-\alpha(\lambda_i)}{ip(j)}
	\leq
	\frac{\lambda_e+|\alpha(\lambda_i)|}{ip(j)}.
	\end{equation}
	Since $p(j)>0$ and $\alpha(\lambda_i)$ is finite, the right-hand side tends to zero as $i\to+\infty$.
	Since the events $A_{i,j}$ decrease to $\{B\in\Sigma:\mathcal{R}B=0\text{ and } |B|<j\}$ as $i\to\infty$, taking this limit in~\eqref{eq_measureAij} gives
		\[\tilde{\mu}\big(\{B\in\Sigma:\mathcal{R}B=0\text{ and }|B|<j\}\big)=0.\]
	As this holds for every $j$, the limiting measure is supported on infinite configurations whose right edge is at the origin, that is, $\tilde{\mu}(\Sigma^{\odot})=1$.
	\end{proof}

	\section{Non-survival at criticality} \label{section_criticaldies}
	
	In this section, we prove Theorem~\ref{thm_criticaldies}. The argument follows the one-dimensional block construction of Bezuidenhout and Grimmett~\cite{BG90}. We keep only the parts that are needed on the line.
	
	The main difference from the classical contact process appears in the use of the graphical construction. In the classical process, once the Poisson marks in a space-time box are fixed, the existence of an active path inside that box is a local event. It does not depend on the configuration outside the box. For the modified boundary contact process this is not true in general, because whether an infection arrow is allowed depends on whether its starting point is a boundary site or an internal site, and this may be affected by particles outside the box. In the attractive region $\lambda_i>\lambda_e$, however, this dependence is monotone: adding particles can only turn boundary sites into internal sites, and the corresponding infection rate increases. This is the reason why the one-dimensional block argument of Bezuidenhout and Grimmett can still be adapted to the present model.
	
	We first record the block estimate used in the proof.
	
	\begin{lemma} \label{lemma_connected_interval}
		Suppose that $\lambda_i>\lambda_e$ and that $\mathbb{P}_{\lambda_i,\lambda_e}(\xi^0_t\neq\emptyset \text{ for all }t)>0$. Then, for every $\varepsilon>0$, there are $R,L\in\mathbb{N}$, $S>0$, and $\delta>0$ such that, for every $\tilde\lambda_i,\tilde\lambda_e$ with $\max\{|\lambda_i-\tilde\lambda_i|,|\lambda_e-\tilde\lambda_e|\}<\delta$, the following event has $\mathbb{P}_{\tilde\lambda_i,\tilde\lambda_e}$-probability at least $1-\varepsilon$: for the process started from $[-R,R]$, there exists $(x_0,t_0)\in[L,2L]\times[S,2S]$ such that every point of $[-R+x_0,R+x_0]\times\{t_0\}$ is reached from $[-R,R]\times\{0\}$ by an active path contained in $[-L,3L]\times[0,2S]$. The same statement holds with the reflected box $[-3L,L]\times[0,2S]$ and $x_0\in[-2L,-L]$.
	\end{lemma}
	
	We defer the proof of the lemma to the end of the section.
	
	\begin{proof}[Proof of Theorem~\ref{thm_criticaldies}.]
		Let $(\xi_t)_{t\geq0}$ be the modified boundary contact process with parameters $(\lambda_i,\lambda_e)$, where $\lambda_i>\lambda_e$. We prove the following openness statement. If $\mathbb{P}_{\lambda_i,\lambda_e}(\xi^0_t\neq\emptyset \text{ for all }t>0)>0$, then there is $\delta>0$ such that, for every $\tilde\lambda_i,\tilde\lambda_e$ satisfying $\max\{|\lambda_i-\tilde\lambda_i|,|\lambda_e-\tilde\lambda_e|\}<\delta$, one has $\mathbb{P}_{\tilde\lambda_i,\tilde\lambda_e}(\xi^0_t\neq\emptyset \text{ for all }t>0)>0$. This proves the theorem, since survival at $(\lambda_i,\lambda^e_*(\lambda_i))$ would then imply survival for some smaller external rate.
		
		Choose $\varepsilon>0$ small enough for the one-dependent oriented percolation comparison below. Let $R,L,S$ and $\delta$ be as in Lemma~\ref{lemma_connected_interval}. We take $\delta$ smaller if necessary so that $\tilde\lambda_i>\tilde\lambda_e$ whenever $\max\{|\lambda_i-\tilde\lambda_i|,|\lambda_e-\tilde\lambda_e|\}<\delta$. Thus the perturbed process remains attractive.
		
		When the construction is restarted from an infected translate of $[-R,R]$, there may be other infected particles already present. We do not remove them. Since $\tilde\lambda_i>\tilde\lambda_e$, adding particles can only turn boundary sites into internal sites, and therefore can only increase the set of allowed infection arrows. Hence the probability of the next block event is bounded from below by the probability given in Lemma~\ref{lemma_connected_interval}.
		
		We now compare with one-dependent oriented percolation on the even sublattice of $\mathbb{Z}\times\mathbb{N}$. A site $(m,n)$ represents the presence, at some time in $[2nS,2(n+1)S]$, of an infected translate of $[-R,R]$ centered in the spatial interval $[mL-L,mL+L]$. From such a block, Lemma~\ref{lemma_connected_interval} gives a right block event and a left block event, each with probability at least $1-\varepsilon$. These two events are increasing events of the underlying Poisson graphical construction. Since this construction is a product measure, they are positively correlated by the FKG inequality. Therefore both events occur with probability at least $(1-\varepsilon)^2$.
		
		The events attached to blocks at distance at least two are independent, because they depend on disjoint space-time boxes. Thus the renormalized process dominates a one-dependent oriented percolation with density as close to one as we wish. For $\varepsilon$ small enough, this one-dependent oriented percolation survives with positive probability. On this event, the modified boundary contact process started from an infected translate of $[-R,R]$ also survives.
		
		Finally, starting from a single infected site, there is a positive probability of producing an infected translate of $[-R,R]$ in finite time. It follows that $\mathbb{P}_{\tilde\lambda_i,\tilde\lambda_e}(\xi^0_t\neq\emptyset \text{ for all }t>0)>0$, which proves the openness statement.
	\end{proof}
	
	\begin{proof}[Proof of Lemma~\ref{lemma_connected_interval}.]
		For $R,L\in\mathbb{N}$ and $S>0$, let $G^+(R,L,S)$ be the event that, for the process started from $[-R,R]$, there exist $x_0\in[L,2L]$ and $t_0\in[S,2S]$ such that every point of $[-R+x_0,R+x_0]\times\{t_0\}$ is reached from $[-R,R]\times\{0\}$ by an active path contained in $[-L,3L]\times[0,2S]$. Let $G^-(R,L,S)$ be the reflected event, with the box $[-3L,L]\times[0,2S]$ and $x_0\in[-2L,-L]$. We prove the estimate for $G^+(R,L,S)$; the proof for $G^-(R,L,S)$ is the same by reflection.
		
		We say that $(x,0)$ is connected to $+\infty$ if, for every $t>0$, it is connected to some point $(y,t)$ by active paths for the process started from the singleton $\{x\}$. The reference to the initial condition is part of the definition, since whether an arrow is allowed may depend on the current infected set.
		
		For $x\in\mathbb{Z}$, let $H_x$ be the event that $(x,0)$ is connected to $+\infty$. The events $H_x$ are spatial translates of $H_0$, and by assumption $\mathbf{P}(H_0)>0$. By spatial ergodicity of the graphical construction, a sufficiently large interval contains, with probability arbitrarily close to one, at least one point $x$ for which $H_x$ occurs. Since $\lambda_i>\lambda_e$, the process is attractive; hence, on the event $\bigcup_{x=-R}^R H_x$, the process started from $[-R,R]$ survives forever. We choose $R$ large enough so that $\mathbb{P}_{\lambda_i,\lambda_e}(\xi^{[-R,R]}_t\neq\emptyset \text{ for all }t\geq0)>1-\varepsilon$.
		
		The only extra point, compared with the classical contact process, is
		locality.  In the modified boundary process, the status of an arrow in a
		space-time box may depend on particles outside the box, since such
		particles can change whether the starting point of the arrow is a
		boundary site.  In the present region $\lambda_i>\lambda_e$, however,
		this dependence is monotone in the useful direction.  Adding particles
		outside the box can only turn a boundary site into an internal site, and
		therefore can only increase the rate attached to an infection arrow.
		Thus the crossing events used below are increasing events of the
		graphical construction, and their probabilities are bounded from below
		by the corresponding probabilities when no outside particles are added.
		
		With this observation, the one-dimensional growth estimate from the proof of~\cite[Lemma 7]{BG90} applies. In dimension one, the estimate only has to produce many endpoints on the top of a large space-time rectangle or many well-separated endpoints on one of its sides. Applied to the present graphical construction, with open paths replaced by active paths for the modified boundary process, it gives $L$ and $T$ such that, with probability at least $1-\varepsilon$, the process from $[-R,R]$ reaches many points on the right side of $[-L,L]\times[0,T]$ and many points on the top of the same box.
		
		Choose $h>0$ small. From any one of the well-separated endpoints on the right side, there is a fixed positive probability, depending only on $R,h,\lambda_i,\lambda_e$, of filling a translate of $[-R,R]$ during the next time interval of length $h$. Taking sufficiently many well-separated endpoints and using independence in the corresponding disjoint boxes, the probability that none of these attempts succeeds can be made smaller than $\varepsilon$.
		
		Combining the previous steps, we can find $L$ and $S=T+h$ such that $\mathbb{P}_{\lambda_i,\lambda_e}(G^+(R,L,S))\geq1-4\varepsilon$. Replacing $\varepsilon$ by a smaller number at the beginning gives the desired bound for the original parameters.
		
		Finally, the event $G^+(R,L,S)$ depends only on finitely many Poisson marks in the finite box $[-L,3L]\times[0,2S]$. Its probability is therefore continuous in the two infection rates. We may choose $\delta>0$ so that the same bound holds for all $(\tilde\lambda_i,\tilde\lambda_e)$ with $\max\{|\lambda_i-\tilde\lambda_i|,|\lambda_e-\tilde\lambda_e|\}<\delta$. Taking $\delta$ smaller if necessary also ensures that $\tilde\lambda_i>\tilde\lambda_e$ throughout this neighbourhood.
	\end{proof}
	
	\bibliographystyle{amsplain}
	\bibliography{refbib.bib}
	
	
	


\end{document}